\documentclass[final,5p,twocolumn]{elsarticle}%

\journal{Systems \& Control Letters}

\usepackage{graphicx}      % include this line if your document contains figures
\usepackage{natbib}        % required for bibliography

\usepackage{bookmark}

\usepackage{epsfig}
\usepackage{amsmath}%

\graphicspath{{fig/}}

\usepackage{amsfonts}
\usepackage{amssymb}

\usepackage{tikz}
\usetikzlibrary{matrix,positioning,decorations.pathreplacing}

\usepackage{color}

\usepackage{enumitem}

\usepackage{ifthen}

\newenvironment{proof}{{\bf Proof.}}{\hfill \hspace*{1pt}\hfill $\Box$}

\makeatletter
\DeclareOldFontCommand{\rm}{\normalfont\rmfamily}{\mathrm}
\DeclareOldFontCommand{\sf}{\normalfont\sffamily}{\mathsf}
\DeclareOldFontCommand{\tt}{\normalfont\ttfamily}{\mathtt}
\DeclareOldFontCommand{\bf}{\normalfont\bfseries}{\mathbf}
\DeclareOldFontCommand{\it}{\normalfont\itshape}{\mathit}
\DeclareOldFontCommand{\sl}{\normalfont\slshape}{\@nomath\sl}
\DeclareOldFontCommand{\sc}{\normalfont\scshape}{\@nomath\sc}
\makeatother

\usepackage{csquotes}  % For \enquotes
\newcommand\q{\enquote}

\usepackage{physics}   %For \curl

\newcommand{\ccat}[3]{{#1\, \underset{#3}{\lozenge}\,{#2}}}

\newcommand{\tm}{\times}

\newcommand \Limsup {\mathop{\overline{\lim}}}

\newcommand \N   {\mathbb{N}}
\newcommand \R   {\mathbb{R}}

\newcommand \K   {\mathcal{K}}
\newcommand \Kinf{\mathcal{K_\infty}}

\newcommand{\Uc}{\ensuremath{\mathcal{U}}}

\newcommand{\Dc}{\ensuremath{\mathcal{D}}}

\newcommand{\vertiii}[1]{{\left\vert\kern-0.25ex\left\vert\kern-0.25ex\left\vert #1 
    \right\vert\kern-0.25ex\right\vert\kern-0.25ex\right\vert}}

\newcommand \qrq   {\quad\Rightarrow\quad}

\newcommand \srs   {\ \ \Rightarrow\ \ }

\newcommand \Iff   {\Leftrightarrow}

\newcommand{\normt}[1]{{\left\vert\kern-0.25ex\left\vert\kern-0.25ex\left\vert #1 
		\right\vert\kern-0.25ex\right\vert\kern-0.25ex\right\vert}}

\newif\ifMath					%For mathematical applications only
\newif\ifEngi					%For engineering applications only

\newif\ifAndo              %Some comments, that are not intended for final version. 
\newif\ifExercises					%Yes, if the exercises are included (e.g. exercises do not fit for Habil, or PhD thesis)
\newif\ifSolutions          %Yes, if the solutions (in this if environment) should be included.
\newif\ifGerman							%Yes, if the German parts should be included
\newif\ifEnglish						%Yes, if the English parts should be included

\newif\ifnothabil						%Yes, if this is NOT a habilitation thesis
\newif\ifFuture							%Yes, if some 'Future', unfinished parts should be added

\newif\ifConf                    %Parts only for a conference version of the paper
\newif\ifJournal								 %Parts only for a journal version of the paper

\newif\ifNOTFORBOOK
\newif\ifFullVersion
\newif\ifExludedDueToSpaceReasons

\usepackage{xifthen}

\newcommand{\einsnorm}[2]{\ensuremath{
    \!\!\;\!\!\!\;
    \left\bracevert\!\!\!\!\!\left\bracevert
    \!
		\ifthenelse{\isempty{#2}}{#1}{#1(#2)}
    \!
      \right\bracevert\!\!\!\!\!\right\bracevert
    \!\!\;\!\!\!\;
  }}

\usepackage{xcolor}
\definecolor{blond}{rgb}{0.98, 0.94, 0.75}
	
\newlength\mytemplen
\newsavebox\mytempbox

\makeatletter
\newcommand\mybluebox{%
    \@ifnextchar[%]
       {\@mybluebox}%
       {\@mybluebox[0pt]}}

\def\@mybluebox[#1]{%
    \@ifnextchar[%]
       {\@@mybluebox[#1]}%
       {\@@mybluebox[#1][0pt]}}

\def\@@mybluebox[#1][#2]#3{
    \sbox\mytempbox{#3}%
    \mytemplen\ht\mytempbox
    \advance\mytemplen #1\relax
    \ht\mytempbox\mytemplen
    \mytemplen\dp\mytempbox
    \advance\mytemplen #2\relax
    \dp\mytempbox\mytemplen
    \colorbox{blond}{\hspace{1em}\usebox{\mytempbox}\hspace{1em}}}

\makeatother
\makeatletter
\let\origd=\d
\renewcommand*\d{
  \relax\ifmmode
    \mathrm{d}%
  \else
    \expandafter\origd
  \fi
}\makeatother

\usepackage{mathtools}

\makeatletter % changes the catcode of @ to 11
\newcommand{\pushright}[1]{\ifmeasuring@#1\else\omit\hfill$\displaystyle#1$\fi\ignorespaces}
\newcommand{\pushleft}[1]{\ifmeasuring@#1\else\omit$\displaystyle#1$\hfill\fi\ignorespaces}
\makeatother % changes the catcode of @ back to 12

\newcounter{syscounter}
\newenvironment{sysnum}{\begin{list}{($\Sigma{\arabic{syscounter}}$)}%
{\settowidth{\labelwidth}{($\Sigma4$)}
\settowidth{\leftmargin}{($\Sigma4$)~}%
\usecounter{syscounter}}}
{\end{list}}

\newcounter{WPcounter}%[section] %[section]

\newtheorem{theorem}{Theorem}[section]
\newtheorem{lemma}[theorem]{Lemma}
\newtheorem{proposition}[theorem]{Proposition}
\newtheorem{corollary}[theorem]{Corollary}
\newtheorem{definition}[theorem]{Definition}

\newtheorem{ass}{Assumption}[section]  

\newtheorem{nnremark}[theorem]{\bf Remark}
\newenvironment{remark}{\begin{nnremark} \rm }{\hfill \hspace*{1pt}\hfill $\lrcorner$\end{nnremark}}

\Andotrue       
\nothabilfalse

\Exercisesfalse
\Solutionsfalse 
\Germanfalse

\Englishtrue

\Futurefalse

\title{\LARGE Lyapunov criteria for robust forward completeness of distributed parameter systems}

\tnotetext[mytitlenote]{The work of A.~Mironchenko is supported by the German Research Foundation (DFG) through the grant MI 1886/2-2.}
 
\author[passauaddress]{Andrii~Mironchenko}
\address[passauaddress]{Faculty of Computer Science and Mathematics, University of Passau, Germany}

\begin{document}

\begin{abstract}
We show that the robust forward completeness for distributed parameter systems is equivalent to the existence of a corresponding Lyapunov function that increases at most exponentially along the trajectories. 
\end{abstract}


\maketitle

\section{Introduction}

A control system is called forward complete if for any initial condition $x$, and any input $u$, the corresponding trajectory $\phi(\cdot,x,u)$ is well-defined on the whole nonnegative time axis. If additionally, for any magnitude $r>0$ and any time $\tau>0$
\[
\sup_{\|x\| \leq r,\ u\in\Dc,\ t\in[0,\tau]} \|\phi(t,x,u)\| <+\infty,
\] 
where $\Dc$ is the space of admissible inputs, then a control system is said to be robustly forward complete (the concept is coined by \cite{Kar04}, but used implicitly at least since \cite{LSW96}).

\emph{Robust forward completeness (RFC)}, as well as a related concept of boundedness of reachability sets, are essential in many contexts. They were instrumental in deriving converse Lyapunov theorems for global asymptotic stability \cite{LSW96}. 
They help to establish regularity properties of the flow maps for (in)finite-dimensional nonlinear systems \cite[Theorem 1.40]{Mir23},  \cite[Section 3.5]{Mir23d}. 
Uniform global asymptotic stability for infinite-dimensional systems has been characterized in terms of uniform weak attractivity, local stability, and RFC property in \cite{MiW19a}. Criteria for input-to-state stability in terms of uniform limit property, local stability, and boundedness of reachability sets were proved for general nonlinear control systems in \cite{MiW18b}. These characterizations, in turn, paved the way for the development of non-coercive Lyapunov methods \cite{MiW19a, MiW18b, JMP20}, characterization of global asymptotic stability for retarded systems \cite{KPC22}, to name a few.

Sufficient conditions for the global existence of solutions for ordinary differential equations (ODEs) and other classes of control systems are a classic subject \cite{Win45, Had72, Jus67, Tan92, Iwa83}. 
For instance, Wintner's theorem \cite{Win45} shows that an ODE 
\[
\dot{x} = f(x)
\]
with locally Lipschitz $f$ has unique global solutions provided that $|f(x)| \leq L(|x|)$ with $L$ satisfying 
\[
\int_c^\infty\frac{1}{L(s)}ds = +\infty\quad \forall c>0.
\]
In particular, if $f$ is globally Lipschitz continuous or linearly bounded, the solutions for the above ODE exist globally, and the reachability sets are bounded (i.e., the system is RFC). This result can be extended to evolution equations in Banach spaces and other system classes, e.g., \cite[Theorem 3.3, p. 199]{Paz83}.

The analysis of necessary conditions for forward completeness is more recent.
Necessary and sufficient conditions of Lyapunov type for forward completeness of ODEs without inputs have been proposed in \cite{KaS67}. However, Lyapunov functions constructed in \cite{KaS67} are time-variant even for time-invariant ODEs.

In \cite{LSW96, AnS99} for systems 
\[
\dot{x} = f(x,u),
\]
with Lipschitz continuous (in both arguments) $f$, it was shown that: forward completeness, boundedness of reachability sets for ODEs with inputs, and the existence of a Lyapunov function that increases at most exponentially, are equivalent properties. 

For distributed parameter systems, the situation is more complex. Linear forward complete infinite-dimensional systems have always bounded reachability sets \cite[Proposition 2.5]{Wei89b}. However, nonlinear forward complete infinite-dimensional systems with Lipschitz continuous right-hand sides do not necessarily have bounded reachability sets, even for systems without inputs, as demonstrated in \cite[Example 2]{MiW18b}.
This fact indicates that the RFC property (establishing uniform bounds for solutions on finite time intervals) is a bridge between the pure well-posedness theory (that studies existence and uniqueness but does not care much about the bounds for solutions) and the stability theory (which is interested in establishing certain bounds for solutions for all nonnegative times, as well as their convergence).

In this work, we consider a broad class of control systems satisfying the so-called boundedness-implies-continuation property and having flows that are Lipschitz continuous on compact intervals.
We show that for this class of systems, \emph{robust forward completeness is equivalent to the existence of a Lyapunov function that increases at most exponentially along the trajectories.}

Our proof differs from that of \cite{AnS99}, where a closely related result was shown for ODE systems.
Namely, for ODEs with Lipschitz right-hand sides, local solutions exist both in a positive and a negative direction. This fact was used for the construction of \q{RFC Lyapunov functions} in \cite{AnS99}. At the same time, for the class of systems that we consider, the solutions backward in time do not necessarily exist, and if they do, then they do not need to be unique. 
To overcome this challenge, we propose a different proof scheme motivated by the converse Lyapunov results for the UGAS property, e.g., \cite[Theorem 4.2.1]{Hen81}.

\textbf{Notation.} 
We write $\N$, $\R$, and $\R_+$ for the sets of positive integers, real numbers, and nonnegative real numbers, respectively.
We say that $\gamma:\R_+\to\R_+$ belongs to the class $\K$ if $\gamma$ is continuous, $\gamma(0)=0$, and $\gamma$ is strictly increasing. 
$\gamma\in\Kinf$ if $\gamma\in\K$ and it is unbounded.

For a normed vector space $S$ we denote the open ball of radius $r$ around $0\in S$ by $B_{r,S}:=\{u\in S:\|u\|_{S}<r\}$. 
If $S$ is the state space $X$, then we denote for short $B_r:=B_{r,X}$.
%By 	$\clo{A}$, we denote the closure of a set $A$ (in a given topology).

\section{General class of systems}

We start with a general definition of a control system.
\index{control system}
\begin{definition}
\label{Steurungssystem}
Consider the triple $\Sigma=(X,\Uc,\phi)$ consisting of 
\index{state space}
\index{space of input values}
\index{input space}
\begin{enumerate}[label=(\roman*)]  
    \item A normed vector space $(X,\|\cdot\|_X)$, called the \emph{state space}, endowed with the norm $\|\cdot\|_X$.
    %\item A \emph{set of input values} $U$, which is a nonempty subset of a certain normed vector space.
    %\item \amc{A normed vector \emph{space of input values} $U$.}
    \item A normed vector \emph{space of inputs} $\Uc \subset \{u:\R_+ \to U\}$          
endowed with a norm $\|\cdot\|_{\Uc}$, where $U$ is a normed vector \emph{space of input values}.
%We assume that the following two axioms hold:
We assume that the following axiom holds:
                    
\emph{The axiom of shift invariance}: for all $u \in \Uc$ and all $\tau\geq0$ the time
shift $u(\cdot + \tau)$ belongs to $\Uc$ with \mbox{$\|u\|_\Uc \geq \|u(\cdot + \tau)\|_\Uc$}.
%the latter is used in the proof of Lemma 6.

%\emph{The axiom of concatenation}: for all $u_1,u_2 \in \Uc$ and for all $t>0$ the \emph{concatenation of $u_1$ and $u_2$ at time $t$}, defined by
%\begin{equation}
%\ccat{u_1}{u_2}{t}(\tau):=
%\begin{cases}
%u_1(\tau), & \text{ if } \tau \in [0,t], \\ 
%u_2(\tau-t),  & \text{ otherwise},
%\end{cases}
%\label{eq:Composed_Input}
%\end{equation}
%belongs to $\Uc$.

    %\item A map $\phi:D_{\phi} \to X$, $D_{\phi}\subseteq \R_+ \times X \times \Uc$ (called \emph{transition map}), so that for all $(x,u)\in X \tm \Uc$ there is a $t>0$ so that $[0,t]\tm \{(x,u)\} \subset D_{\phi}$.
    \item A map $\phi:D_{\phi} \to X$, $D_{\phi}\subseteq \R_+ \times X \times \Uc$ (called \emph{transition map}), such that for all $(x,u)\in X \tm \Uc$ it holds that $D_{\phi} \cap \big(\R_+ \times \{(x,u)\}\big) = [0,t_m)\tm \{(x,u)\} \subset D_{\phi}$, for a certain $t_m=t_m(x,u)\in (0,+\infty]$.
		
		The corresponding interval $[0,t_m)$ is called the \emph{maximal domain of definition} of $t\mapsto \phi(t,x,u)$.
		
\end{enumerate}
The triple $\Sigma$ is called a \emph{(control) system}, if the following properties hold:
\index{property!identity}

\begin{sysnum}
    \item\label{axiom:Identity} \emph{The identity property:} for every $(x,u) \in X \times \Uc$
          it holds that $\phi(0, x,u)=x$.
\index{causality}
    \item \emph{Causality:} for every $(t,x,u) \in D_\phi$, for every $\tilde{u} \in \Uc$, such that $u(s) =
          \tilde{u}(s)$ for all $s \in [0,t]$ it holds that $[0,t]\tm \{(x,\tilde{u})\} \subset D_\phi$ and $\phi(t,x,u) = \phi(t,x,\tilde{u})$.
    \item \label{axiom:Continuity} \emph{Continuity:} for each $(x,u) \in X \times \Uc$ the map $t \mapsto \phi(t,x,u)$ is continuous on its maximal domain of definition.
\index{property!cocycle}
        \item \label{axiom:Cocycle} \emph{The cocycle property:} for all
                  $x \in X$, $u \in \Uc$, for all $t,h \geq 0$ so that $[0,t+h]\tm \{(x,u)\} \subset D_{\phi}$, we have
\[
\phi\big(h,\phi(t,x,u),u(t+\cdot)\big)=\phi(t+h,x,u).
\]
\end{sysnum}

\end{definition}

Definition~\ref{Steurungssystem} can be viewed as a direct generalization and a unification of the concepts of strongly continuous nonlinear semigroups with abstract linear control systems \cite{Wei89b}. 
This class of systems encompasses control systems generated by ODEs, switched systems, time-delay systems,
evolution partial differential equations, differential
equations in Banach spaces and many others \cite[Chapter 1]{KaJ11b}.

For a wide class of control systems, the boundedness of a solution implies the possibility of prolonging it to a larger interval, see \cite[Chapter 1]{KaJ11b}. Next, we formulate this property for abstract systems:
\begin{definition}
\label{def:BIC} 
\index{property!boundedness-implies-continuation}
\index{BIC}
We say that a system $\Sigma$ satisfies the \emph{boundedness-implies-continuation (BIC) property} if for each
$(x,u)\in X \tm \Uc$ such that the maximal existence time $t_m(x,u)$ is finite, and for all $M > 0$, there exists $t \in [0,t_m(x,u))$ with $\|\phi(t,x,u)\|_X>M$.
\end{definition}

Take any $R \in \R_+\cup\{\infty\}$, and assume that the inputs are restricted to the set
\begin{align}
\label{eq:disturbance-set}
\Dc:=\{u \in\Uc: \|u\|_\Uc \leq R\}.
\end{align}

\begin{definition}
\label{def:FC_Property} 
\index{forward completeness}
We say that a control system $\Sigma=(X,\Uc,\phi)$ is \emph{forward complete (for inputs in $\Dc$)}, if 
$\R_+ \tm X\tm\Dc \subset D_\phi$, that is for every $(x,u) \in X \times \Dc$ and for all $t \geq 0$ the value $\phi(t,x,u) \in X$ is well-defined.
\end{definition}

Forward completeness alone does not imply, in general, the existence of any uniform bounds on the trajectories emanating from bounded balls, even in the absence of inputs \cite[Example 2]{MiW18b}.
If the system does exhibit such a bound, it deserves a special name (the term was first introduced in \cite[Definition 2.2]{Kar04}, though implicitly it was used before, e.g., in \cite[Corollary 2.4]{AnS99}).
%The following variation of the BRS property is important in the context when $u$ plays a role of a disturbance rather than of an input, see \cite[Definition 2.1]{KaJ11b}:
\begin{definition}
\label{def:RFC}
\index{robust forward completeness}
\index{RFC}
Consider a control system $\Sigma=(X,\Uc,\phi)$, and let $\Dc$ be the set of disturbances as defined by \eqref{eq:disturbance-set}.
We say that \emph{$\Sigma$ is robustly forward complete (RFC)} if $\Sigma$ is forward complete for inputs in $\Dc$, and for any $r>0$ and any $\tau>0$, it holds that 
\[
\sup\big\{
\|\phi(t,x,u)\|_X : \|x\|_X\leq r,\ u\in\Dc,\ t \in [0,\tau]\big\} < \infty.
\]
\end{definition}
%Note that above the supremum is taken over all inputs in $\Dc$. 
%If $R$ is finite, then the BRS property clearly implies RFC. At the same time, if $R=+\infty$ (and thus $\Dc=\Uc$), RFC is essentially stronger than the BRS property.

For ODEs with Lipschitz continuous right-hand side, forward completeness is always robust provided that $R<\infty$, as shown in \cite[Proposition 2.5]{LSW96}. However, if $\Dc=\Uc$, robust forward completeness is, in general, essentially stronger than forward completeness, even for scalar systems.

In particular, the scalar system  $\dot{x} = x u$ with $X=\R$ and $\Dc=\Uc:=L^\infty(\R_+,\R)$ is forward complete, but it is not RFC. A simple example of a scalar RFC system is given by the following system with the same $X,\Uc,\Dc$:
\[
\dot{x} = \frac{1}{1+|u(t)|} x.
\]

\section{Criteria for robust forward completeness}
\label{sec:Criteria for robust forward completeness}

In this section, we derive Lyapunov criteria for RFC.

\subsection{Technical lemmas}

We need the following lemma (\cite[Lemma B.29]{Mir23}):
\begin{lemma}
\label{lem:SimpleIneqLemma}
Let $f,g:D \to \R_+$ be any functions for which $\sup_{d \in D}f(d)$ is finite.
Then
\begin{eqnarray}
\sup_{d \in D}f(d) - \sup_{d \in D}g(d) \leq \sup_{d \in D}(f(d) - g(d)).
\label{eq:SimpleIneqLemma}
\end{eqnarray}
\end{lemma}

Another technical ingredient that we will need is:
\begin{lemma}
\label{Converse_Lyapunov_Theorem_SimpleEstimate}
For any $k \in \N$, consider the function 
\begin{eqnarray}
G_k: r \mapsto \max\Big\{r-\frac{1}{k},0 \Big\}.
\label{eq:G_k}
\end{eqnarray}
Then
\begin{enumerate}[label = (\roman*)]
	\item $G_k$ is Lipschitz continuous with a unit Lipschitz constant, i.e., for all $r_1,r_2 \geq 0$, it holds that 
\begin{eqnarray}
\big|G_k(r_1) - G_k(r_2)\big| \leq |r_1-r_2|.
\label{eq:CLT_Simp_Est}
\end{eqnarray}
	\item For any $a\geq 1$ and any $k\in\N$, it holds that
\begin{eqnarray}
G_k(ar) \leq a G_k(r)+ \frac{a-1}{k},\quad r\geq 0.
\label{eq:G_k-subadditivity}
\end{eqnarray}
\end{enumerate}
\end{lemma}

\begin{proof}
(i) holds as each $G_k$ is the maximum of two Lipschitz continuous functions with Lipschitz constant at most 1.

(ii). For $a\ge 1$ we have that 
\begin{align*}
%\label{eq:}
G_k(ar) 
&= \max\Big\{ar-\frac{1}{k},0 \Big\} \\
&= \max\Big\{ar-a\frac{1}{k} + \frac{a-1}{k},0 \Big\} \\
&\leq \max\Big\{ar-a\frac{1}{k} ,0 \Big\} + \frac{a-1}{k}\\
&= aG_k(r) + \frac{a-1}{k}.
\end{align*}
\end{proof}

The following lemma is taken from \cite[p.130]{KaJ11b} (where it was stated informally); see also \cite[Lemma A.18]{Mir23}, where a more general result is stated.
\begin{lemma}
\label{lem:KinfLipschitzLowerEstimate}
For any $\alpha \in \Kinf$, there exists $\rho \in \Kinf$ so that $\rho(s) \leq \alpha(s)$ for all $s \in \R_+$ and $\rho$ is globally Lipschitz with a unit Lipschitz constant, i.e. for any $s_1,s_2 \geq 0$ it holds that 
\begin{equation}
|\rho(s_1) - \rho(s_2)| \leq |s_1 - s_2|.
\label{eq:Lipschitz_Kinf}
\end{equation}
\end{lemma}

For a continuous function $y:\R \to \R$, let
the \emph{right upper Dini derivative} be defined by
$D^+y(t):=\Limsup\limits_{h \to +0}\frac{y(t+h)-y(t)}{h}$.

\begin{proposition}[Comparison principle]
\label{prop:ComparisonPrinciple-exp}
For any $y \in C(\R_+, \R_+)$ satisfying the differential inequality
\begin{eqnarray}
D^+y(t) \leq a y(t) + M \quad \forall t>0,
\label{eq:ComparisonPrinciple-exp}
\end{eqnarray}
with certain $a,M>0$, it holds that
\begin{eqnarray}
y(t) \leq y(0)e^{at} + \frac{M}{a}(e^{at}-1) \quad \forall t \geq 0.
\label{eq:ComparisonPrinciple_FinalEstimate-exp}
\end{eqnarray}
\end{proposition}

\begin{proof}
Define $z(t):=y(t) - \frac{M}{a}(e^{at} - 1)$, $t\geq 0$.
Then 
\begin{align*}
D^+z(t) 
&= D^+y(t) - Me^{at}\\
&\le a\Big(z(t) + \frac{M}{a}(e^{at} - 1)\Big) + M - Me^{at}\\
&= az(t).
\end{align*}
Arguing as in \cite[The proof of Lemma~3.2, p. 464]{MiI16}, we obtain the counterpart of the estimate 
\cite[eq. (41)]{MiI16} for $z$:
\[
z(t) \leq \eta^{-1}(\eta(z(0))+at),
\]
with $h(s)=\ln(s)$, which implies that $z(t)\leq z(0)e^{at}$, and thus \eqref{eq:ComparisonPrinciple_FinalEstimate-exp} holds.
\end{proof}

\subsection{Lyapunov characterization of RFC}

We call a function $h: \R_+^3 \to \R_+$ increasing, if $(r_1,r_2,r_3) \leq (R_1,R_2,R_3)$
implies that $h(r_1,r_2,r_3) \leq h(R_1,R_2,R_3)$, where we use the component-wise
partial order on $\R_+^3$.
 %We call $h$ strictly increasing if $(r_1,r_2,r_3)
%\leq (R_1,R_2,R_3)$ and $(r_1,r_2,r_3) \neq (R_1,R_2,R_3)$ imply $h(r_1,r_2,r_3) <
%h(R_1,R_2,R_3)$.

The regularity of Lyapunov functions, constructed via converse Lyapunov techniques, depends on the regularity of the flow map.
\begin{definition}
\label{def:Lipschitz-uniform-in-inputs}
Let $\Dc$ be the set of disturbances as defined by \eqref{eq:disturbance-set}.
We say that the flow of a control system $\Sigma=(X,\Uc,\phi)$ is \emph{Lipschitz continuous on compact intervals uniformly in inputs from $\Dc$}, if 
for any $\tau>0$ and any $r>0$, there exists $L = L(\tau,r)>0$ so that 
\begin{align}
x,y&\in B_r,\ t \in [0,\tau],\ u\in \Dc \nonumber\\
& \qrq \|\phi(t,x,u) - \phi(t,y,u) \|_X \leq L \|x-y\|_X.
\label{eq:Flow_is_Lipschitz-uniform-in-u}
\end{align}	
\end{definition}

We assume that the axiom of concatenation is valid for the inputs in $\Dc$. 
\begin{ass}[The axiom of concatenation]
\label{ass:Concatenation-in-D}
We suppose that for all $u_1,u_2 \in \Dc$ and for all $t>0$ the \emph{concatenation of $u_1$ and $u_2$ at time $t$}, defined by
\begin{equation}
\ccat{u_1}{u_2}{t}(\tau):=
\begin{cases}
u_1(\tau), & \text{ if } \tau \in [0,t], \\ 
u_2(\tau-t),  & \text{ otherwise},
\end{cases}
\label{eq:Composed_Input}
\end{equation}
belongs to $\Dc$.
\end{ass}

\begin{remark}
\label{rem:Concatenation-in-D} 
If $\Dc=\Uc$, Assumption~\ref{ass:Concatenation-in-D} is satisfied for most of the standard input spaces. 
If $\Dc$ is a bounded ball (i.e., $R$ in \eqref{eq:disturbance-set} is finite), Assumption~\ref{ass:Concatenation-in-D} becomes more restrictive as then the norm of the concatenation of two inputs from $\Dc$ cannot exceed the maximum of the norm of the concatenated inputs. 
In particular, Assumption~\ref{ass:Concatenation-in-D} is valid with $R<\infty$ if $U$ is a Banach space, and $\Uc = L^\infty(\R_+,U)$ (the space of essentially bounded strongly measurable $U$-valued functions), if $\Uc = C_b(\R_+,U)$ (the space of continuous bounded $U$-valued functions), or if $\Uc = PC_b(\R_+,U)$ (the space of piecewise continuous bounded $U$-valued functions). 
At the same time, Assumption~\ref{ass:Concatenation-in-D} is not valid with finite $R$ for $\Uc = L^p(\R_+,U)$, $p\in[1,+\infty)$ (the space of strongly measurable functions $u:\R_+\to U$ such that $s\mapsto \|u(s)\|_U^p$ is  Lebesgue integrable).
\end{remark}

Next, we characterize the RFC property in Lyapunov terms %in a fashion of Theorem~\ref{thm:BRS-characterization}.

\begin{theorem}[Criteria for RFC property]
\label{thm:RFC-characterization} 
	Consider a control system $\Sigma=(X,\Uc,\phi)$.
Let $\Dc$ be the set of disturbances as defined by \eqref{eq:disturbance-set} satisfying Assumption~\ref{ass:Concatenation-in-D}.
Let $\Sigma$ satisfy the BIC property and have a flow that is Lipschitz continuous on compact intervals uniformly in inputs from $\Dc$. 

The following statements are equivalent:
\begin{enumerate}[label=(\roman*)]
	\item $\Sigma$ is robustly forward complete.
	\item  There exists a continuous, increasing function $\mu: \R_+^2 \to \R_+$, such that for
all $x\in X, u\in \Dc$ and all $t \geq 0$ we have
 \begin{equation}
    \label{eq:RFC}
    \| \phi(t,x,u) \|_X \leq \mu( \|x\|_X,t).
\end{equation}
	\item There exists a continuous function $\mu: \R_+^2 \to \R_+$ such that for
all $x\in X, u\in \Dc$ and all $t \geq 0$ the inequality \eqref{eq:RFC} holds.

	\item There are $\xi\in\Kinf$ and $c>0$, such that for all $t\geq 0,\ x \in X,\ u\in\Dc$
\begin{align}
\label{eq:RFC-formulation}
\|\phi(t,x,u)\|_X\leq \xi(\|x\|_X) + \xi(t) + c.
\end{align}
\item There are a Lipschitz continuous function $V:X\to\R_+$, maps $\psi_1,\psi_2 \in\Kinf$, and $C>0$ such that 
\begin{eqnarray}
\psi_1(\|x\|_X) \leq V(x) \leq \psi_2(\|x\|_X) + C,\quad x\in X,
\label{eq:RFC-sandwich-bounds}
\end{eqnarray}
and there are $a, M>0$, such that for all $x \in X$ and $u\in\Dc$, the following holds:
\begin{align}
\label{eq:RFC-dissipation}
\dot{V}_u(x) \leq a V(x) + M,
\end{align}
where $\dot{V}_u(x)$ denotes the right upper Dini derivative of the map $t \mapsto V(\phi(t,x,u))$ at $t=0$, i.e.,
\begin{equation}
\label{ISS_LyapAbleitung}
\dot{V}_u(x):=\Limsup \limits_{t \rightarrow +0} {\frac{1}{t}\big(V(\phi(t,x,u))-V(x)\big) }.
\end{equation}
\end{enumerate}
\end{theorem}

\begin{proof}
\textbf{(i) $\Iff$ (ii) $\Iff$ (iii).} This was shown in \cite[Lemma 2.12]{MiW19a}.

\textbf{(ii) $\Rightarrow$ (iv).} 
Define $\zeta(r):=\mu(r,r) + r$, $r\geq 0$. 
As $\mu$ is increasing and continuous, $\zeta$ is strictly increasing and continuous, and for all $x,u,t$, we have 
\begin{align*}
\mu( \|x\|_X,t) 
&\leq \mu( \|x\|_X, \|x\|_X) + \mu(t,t)  \\
&\leq \zeta(\|x\|_X)  + \zeta(t).
\end{align*}
Define $\xi(r):=\zeta(r) - \lim_{r\to +0}\zeta(r)$. 
Then $\xi\in\Kinf$, and \eqref{eq:RFC-formulation} holds with this $\xi$ and $c:=2\lim_{r\to +0}\zeta(r)$.

%%%%%%%%%%%%%%%%%%%%%%%%%%%%%%%%%%%%%%%%%%%%%%%%%%%%%%%%%%%%%%%%%%%%%%%%%%%

\textbf{(v) $\Rightarrow$ (i).} Pick any $x \in X$ and any $u\in\Dc$. 
As $\Sigma$ is a well-defined control system, there is a maximal time $t_m(x,u)$ such that $\phi(\cdot,x,u)$ is well-defined on $[0,t_m(x,u))$.
By the axiom of shift invariance, $u(\cdot+r)\in \Uc$ and $\|u\|_{\Uc} \geq \|u(\cdot+r)\|_{\Uc}$ for any $r \geq 0$. 
This shows that also $u(\cdot+r)\in \Dc$, and thus the set $\Dc$ is invariant w.r.t.\ time-shift of the input signal as well.
Thus, by \eqref{eq:RFC-dissipation}, for all $t\in [0,t_m(x,u))$ we have that 
\begin{align*}
%\label{eq:RFC-dissipation-with-time}
 D^+V(\phi(t,x,u)) =\dot{V}_{u(\cdot+t)}(\phi(t,x,u)) \leq a V(\phi(t,x,u))+M.
\end{align*}
Employing Proposition~\ref{prop:ComparisonPrinciple-exp} for the continuous map 
\[
y(t):=V(\phi(t,x,u)),\quad t\in[0,t_m(x,u)),
\]
we obtain:
\[
V(\phi(t,x,u)) \leq e^{at}V(x)+ \frac{M}{a}(e^{at}-1),\quad  t\in[0,t_m(x,u)).
\]
Thanks to \eqref{eq:RFC-sandwich-bounds}, we have for all $t\in[0,t_m(x,u))$ that 
\begin{align*}
%\label{eq:}
\psi_1(\|\phi(t,x,u)\|_X) \leq e^{at} \big(  \psi_2(\|x\|_X) + C\big)+ \frac{M}{a}(e^{at}-1),
\end{align*}
and thus 
\begin{align}
\label{eq:RFC-bound}
\|\phi(&t,x,u)\|_X \nonumber\\
&\leq \psi_1^{-1}\Big(e^{at} \big(  \psi_2(\|x\|_X) + C\big)+ \frac{M}{a}(e^{at}-1)\Big).
\end{align}
Now, if $t_m(x,u)$ is finite, the trajectory $\phi(\cdot,x,u)$ is uniformly bounded on $[0,t_m(x,u))$, and we obtain a contradiction to the BIC property. Hence, $t_m(x,u)=+\infty$, and \eqref{eq:RFC-bound} shows the robust forward completeness.

%%%%%%%%%%%%%%%%%%%%%%%%%%%%%%%%%%%%%%%%%%%%%%%%%%%%%%%%%%%%%%%%%%%%%%%%%%%

\textbf{(iv) $\Rightarrow$ (v).} This implication (converse Lyapunov result) will be proved in several steps. 

\textbf{Construction of \q{pre-Lyapunov functions} $V_k$.} Let $\xi\in\Kinf$ be as in (iv). Pick $\rho \in\Kinf$ such that $\rho\leq \xi^{-1}$ pointwise and $\rho$ is globally Lipschitz continuous with a unit Lipschitz constant. Such $\rho$ exists in view of Lemma~\ref{lem:KinfLipschitzLowerEstimate}.

For any $k\in \N$, consider $V_k:X\to \R_+$, defined for all $x\in X$ as follows:
\begin{align}
\label{eq:LF-RFC}
V_k(x):=\sup_{u \in \Dc}\sup_{t\geq 0}G_k\Big(e^{-t}\rho\Big(\frac{1}{3}\|\phi(t,x,u)\|_X\Big)\Big),
\end{align}
where $G_k$ are given by \eqref{eq:G_k}.

To upperestimate $V_k$, recall that for all $\alpha\in\Kinf$ and all $a,b,c\in\R_+$ it holds that
\begin{eqnarray}
\alpha(a+b+c) \leq \alpha(3a)+ \alpha(3b)+\alpha(3c).
\label{eq:triangle-for-Kinf-3}
\end{eqnarray}

Take any $x \in X$. Using in \eqref{eq:LF-RFC} the estimate \eqref{eq:RFC-formulation}, and the fact that $\rho\leq \xi^{-1}$ and $G_k(r)\leq r$ pointwise, we have:
\begin{align*}
%\label{eq:}
V_k(x) \leq \sup_{u \in \Dc}\sup_{t\geq 0}e^{-t}\xi^{-1}\Big(\frac{1}{3}\Big(\xi(\|x\|_X)  + \xi(t) + c\Big)\Big).
\end{align*}
Applying \eqref{eq:triangle-for-Kinf-3} with $\alpha:=\xi^{-1}$, we obtain for all $x \in X$:
\begin{align}
\label{eq:Sandwich-BRS-LF-above}
V_k(x) &\leq \sup_{t\geq 0}e^{-t}\big(\|x\|_X + t+ \xi^{-1}(c) \big) \leq \|x\|_X + C,
\end{align}
for a certain constant $C>0$ (depending solely on $c$).

%%%%%%%%%%%%%%%%%%%%%%%%%%%%%%%%%%%%%%%%%%%%%%%%%%%%%%%%%%%%%%%%%%%%%%%%%%%

\textbf{Growth estimate \eqref{eq:RFC-dissipation} for $V_k$ with $a=1$.} 
Take any $v \in\Dc$, and any $h>0$. 
By the cocycle property, we have
\begin{align*}
%\label{eq:}
V_k(&\phi(h,x,v)) \\
&= \sup_{u \in \Dc}\sup_{t\geq 0}G_k\Big(e^{-t}\rho\Big(\frac{1}{3}\|\phi(t,\phi(h,x,v),u)\|_X\Big)\Big)\\
&= \sup_{u \in \Dc}\sup_{t\geq 0}G_k\Big(e^{-t}\rho\Big(\frac{1}{3}\big\|\phi(t+h,x,\ccat{v}{u}{h})\big\|_X\Big)\Big).
\end{align*}
Here, the concatenation $\ccat{v}{u}{h}$ was defined in \eqref{eq:Composed_Input}.

Assumption~\ref{ass:Concatenation-in-D} ensures that $\ccat{v}{u}{h} \in \Dc$.
Thus, we only increase the right-hand side by taking the supremum over a larger space of inputs:
\begin{align*}
%\label{eq:}
V_k&(\phi(h,x,v)) \\
&\leq \sup_{u \in \Dc}\sup_{t\geq 0}G_k\Big(e^{-t}\rho\Big(\frac{1}{3}\|\phi(t+h,x,u)\|_X\Big)\Big)\\
&= \sup_{u \in \Dc}\sup_{t\geq 0}G_k\Big(e^he^{-(t+h)}\rho\Big(\frac{1}{3}\|\phi(t+h,x,u)\|_X\Big)\Big).
\end{align*}
Applying Lemma~\ref{Converse_Lyapunov_Theorem_SimpleEstimate}(ii), we proceed to
\begin{align*}
%\label{eq:}
V_k&(\phi(h,x,v)) \\
&\leq e^h\sup_{u \in \Dc}\sup_{t\geq 0}G_k\Big(e^{-(t+h)}\rho\Big(\frac{1}{3}\|\phi(t+h,x,u)\|_X\Big)\Big) + \frac{e^h-1}{k}\\
&\leq e^h V_k(x) + \frac{e^h-1}{k}.
\end{align*}
Thus, for all $v \in\Dc$, we have 
\begin{align*}
%\label{eq:}
\dot{V}_{k,v}(x) 
&:=\limsup_{h\to+0}\frac{1}{h}\Big( V_k(\phi(h,x,v)) - V_k(x) \Big)\\ 
&\le \limsup_{h\to+0}\frac{1}{h}(e^h - 1)V_k(x)  + \lim_{h\to0}\frac{e^h-1}{kh} \\ 
&= V_k(x)+\frac{1}{k}.
\end{align*}

%%%%%%%%%%%%%%%%%%%%%%%%%%%%%%%%%%%%%%%%%%%%%%%%%%%%%%%%%%%%%%%%%%%%%%%%%%%

\textbf{Lipschitz continuity for $V_k$ on bounded balls.} 
Take any $r>0$.
Arguing as in \eqref{eq:Sandwich-BRS-LF-above}, we see that for all $x \in B_r$, all $u\in \Dc$, and all $t\geq 0$, it holds that 
\[
e^{-t}\rho\Big(\frac{1}{3}\|\phi(t,x,u)\|_X\Big)\leq e^{-t}\big(r + t+ \xi^{-1}(c) \big).
\]
Hence for any $k \in\N$, there is a time $T=T(r,k)$:
%
%Pick any $R>0$ and $x \in X$ with $\|x\|_X \leq R$. Then for any $d \in \Dc$ and
%for any $t \geq T(R,k) := \ln(1 + k \alpha_1(R))$ it holds that
\begin{equation*}
t \geq T(r,k) \srs e^{-t}\rho\Big(\frac{1}{3}\|\phi(t,x,u)\|_X\Big) \leq \frac{1}{k},\ \  x \in B_r,\  u\in \Dc.   
\end{equation*}
Thus, the domain of maximization in the definition of $V_k$ has a finite length.
That is, for all $r >0$ and all $x \in B_r$, the function $V_k$ can be equivalently defined by
\begin{eqnarray*}
\hspace{-5mm}V_k(x)= \sup_{u \in \Dc}\sup_{t \in[0,T(r,k)]}G_k\Big(e^{-t}\rho\Big(\frac{1}{3}\|\phi(t,x,u)\|_X\Big)\Big).
\end{eqnarray*}

Now pick any $x,y \in B_r$, and consider
\begin{align*}
|V_k(x)-&V_k(y)|  \\
&= \Big|\sup_{u \in \Dc}\sup_{t \in[0,T(r,k)]}G_k\Big(e^{-t}\rho\Big(\frac{1}{3}\|\phi(t,x,u)\|_X\Big)\Big)  \\
&\qquad   -\sup_{u \in \Dc}\sup_{t \in[0,T(r,k)]}G_k\Big(e^{-t}\rho\Big(\frac{1}{3}\|\phi(t,y,u)\|_X\Big)\Big) \Big| 
\end{align*}
%The domains for taking the supremum in these two expressions are different. 
%Assume that $V_k(x)>V_k(y)$. Then if $\|x\|_X\leq\|y\|_X$, we can upperestimate $V_k(x)$ as 
%\[
%V_k(x) \leq \sup_{\|u\|_\Uc \leq \|y\|_X}\sup_{t \in[0,T(R,k)]}G_k\Big(e^{-t}\rho\Big(\frac{1}{3}\|\phi(t,x,u)\|_X\Big)\Big).
%\]
%Otherwise, if $\|x\|_X\geq\|y\|_X$, then 
%\[
%V_k(y) \geq \sup_{u \in \Dc}\sup_{t \in[0,T(R,k)]}G_k\Big(e^{-t}\rho\Big(\frac{1}{3}\|\phi(t,y,u)\|_X\Big)\Big).
%\]
%The case if $V_k(x)<V_k(y)$ can be treated similarly. 
%In any case, we have that 
%\begin{align*}
%&|V_k(x)-V_k(y)|  \\
%&\leq \Big|\sup_{\|u\|_\Uc \leq Q}\sup_{t \in[0,T(R,k)]}G_k\Big(e^{-t}\rho\Big(\frac{1}{3}\|\phi(t,x,u)\|_X\Big)\Big)  \\
%&\qquad\quad   -\sup_{\|u\|_\Uc \leq Q}\sup_{t \in[0,T(R,k)]}G_k\Big(e^{-t}\rho\Big(\frac{1}{3}\|\phi(t,y,u)\|_X\Big)\Big) \Big|, 
%\end{align*}
%where either $Q=\|x\|_X$, or $Q=\|y\|_X$.

Using Lemma~\ref{lem:SimpleIneqLemma}, we proceed to 
\begin{align*}
|V_k(x)&-V_k(y)|  \\
&\leq \sup_{u \in \Dc}\sup_{t \in[0,T(r,k)]}\Big|G_k\Big(e^{-t}\rho\Big(\frac{1}{3}\|\phi(t,x,u)\|_X\Big)\Big)  \\
&\qquad\qquad\qquad   -G_k\Big(e^{-t}\rho\Big(\frac{1}{3}\|\phi(t,y,u)\|_X\Big)\Big) \Big|, 
\end{align*}
As $G_k$ is globally Lipschitz with unit Lipschitz constant, we continue the estimates as follows:
\begin{align*}
|V_k(x)-V_k(y)|  
&\leq \sup_{u \in \Dc}\sup_{t \in[0,T(r,k)]}e^{-t}\Big|\rho\Big(\frac{1}{3}\|\phi(t,x,u)\|_X\Big)\\
&\qquad\qquad\qquad\qquad -\rho\Big(\frac{1}{3}\|\phi(t,y,u)\|_X\Big) \Big|.
\end{align*}
As $\rho$ is also globally Lipschitz with unit Lipschitz constant, we proceed to
\begin{align*}
&|V_k(x)-V_k(y)|  \\
&\leq \frac{1}{3}\sup_{u \in \Dc}\sup_{t \in[0,T(r,k)]}e^{-t}\Big|\|\phi(t,x,u)\|_X - \|\phi(t,y,u)\|_X \Big|\\
&\leq \frac{1}{3}\sup_{u \in \Dc}\sup_{t \in[0,T(r,k)]}\big\|\phi(t,x,u) - \phi(t,y,u) \big\|_X.
\end{align*}
Since $\phi$ is Lipschitz continuous on compact intervals uniformly in inputs in $\Dc$, there is some $M=M(r,k)$, which we assume without loss of generality to be increasing with respect to both arguments, such that: 
\begin{align}
\label{eq:Lipschitz-for-V_k}
|V_k(x)-V_k(y)|  \leq M(r,k)\|x-y \|_X,\quad x,y \in B_r.
\end{align}

%%%%%%%%%%%%%%%%%%%%%%%%%%%%%%%%%%%%%%%%%%%%%%%%%%%%%%%%%%%%%%%%%%%%%%%%%%%

\textbf{Defining \q{RFC Lyapunov function}.} 
Setting in \eqref{eq:LF-RFC} $t:=0$, and using the identity axiom of $\Sigma$, we estimate $V_k$ from below for all $x \in X$ as
\begin{align*}
%\label{eq:}
V_k(x) \geq G_k\Big(\rho\big(\frac{1}{3}\|\phi(0,x,0)\|_X\big)\Big) = G_k\circ \rho\big(\frac{1}{3}\|x\|_X\big).
\end{align*}
As, $G_k(r)>0$ for $r>\frac{1}{k}$, we have
\begin{eqnarray}
\rho\Big(\frac{1}{3}\|x\|_X\Big) >\frac{1}{k} \qrq V_k(x) >0.
\label{eq:Condition-for-V_k-positivity}
\end{eqnarray}
At the same time, if $\rho\big(\frac{1}{3}\|x\|_X\big) <\frac{1}{k}$, we do not have a coercive estimate from below for $V_k$. 
Motivated by \cite[p. 133]{KaJ11b}, and using the Lipschitz constants $M(r,k)$ from \eqref{eq:Lipschitz-for-V_k}, we define a Lyapunov function candidate $W: X\to \R_+$ by
\begin{eqnarray}
W (x) := \sum_{k=1}^{\infty} \frac{2^{-k}}{1+M(k,k)} V_k (x) \qquad \forall x \in X.
\label{eq:LF_ConverseLyapTheorem_integralType_W-LF}
\end{eqnarray}
We have 
\begin{align*}
%\label{eq:Sandwich-for-W-BRS-theorem}
\psi_1(\|x\|_X)&:=\sum_{k=1}^{\infty} \frac{2^{-k}}{1+M(k,k)} G_k\circ \rho\big(\frac{1}{3}\|x\|_X\big) \\
& \leq W (x) \leq  \|x\|_X + C, \quad x \in X.
\end{align*}
Clearly, $\psi_1(0)=0$. Since for each $x \neq 0$ there is some $k\in\N$ such that $\rho\big(\frac{1}{3}\|x\|_X\big) >\frac{1}{k} $, the condition \eqref{eq:Condition-for-V_k-positivity} ensures that $\psi_1(r)>0$ for $r>0$. 
Furthermore, for any $r,s\geq 0$ we have 
\begin{align*}
%\label{eq:}
&|\psi_1(r)-\psi_1(s)| \\
&= \Big| \sum_{k=1}^{\infty} \frac{2^{-k}}{1{+}M(k,k)} G_k\circ \rho\big(\frac{1}{3}r\big) - \sum_{k=1}^{\infty} \frac{2^{-k}}{1{+}M(k,k)} G_k\circ \rho\big(\frac{1}{3}s\big) \Big| \\
&\leq  \sum_{k=1}^{\infty} \frac{2^{-k}}{1+M(k,k)} \Big|G_k\circ \rho\big(\frac{1}{3}r\big) - G_k\circ \rho\big(\frac{1}{3}s\big) \Big|. \end{align*}
As both $G_k$, $k\in\N$, and $\rho$ are globally Lipschitz with unit Lipschitz constant, we proceed to 
\begin{align*}
%\label{eq:}
|\psi_1(r)-\psi_1(s)| 
&\leq \frac{1}{3} \sum_{k=1}^{\infty} \frac{2^{-k}}{1+M(k,k)} \big|r-s\big| \leq \frac{1}{3} \big|r-s\big|,
\end{align*}
which shows the global Lipschitz continuity of $\psi_1$. Finally, as $\rho$ is increasing to infinity, $\psi_1$ shares this property. 
Overall, $\psi_1\in\Kinf$.

Now pick any $r >0$ and any $x,y \in B_r$. Exploiting \eqref{eq:Lipschitz-for-V_k}, we have\begin{align*}
\big|W(x) - W (y)\big| &= \Big| \sum_{k=1}^{\infty} \frac{2^{-k}}{1+M(k,k)} \big(V_k (x) - V_k (y)\big)\Big| \\
&\leq  \sum_{k=1}^{\infty} \frac{2^{-k}M(r,k)}{1+M(k,k)} \|x-y\|_X\\
&\leq  \Big(1+ \sum_{k=1}^{[r] + 1} \frac{2^{-k}M(r,k)}{1+M(k,k)}\Big) \|x-y\|_X.
\end{align*}
This shows that $W$ is a Lyapunov function for $\Sigma$ in the sense of (v), which is Lipschitz continuous on bounded balls.

Differentiating $W$ along the trajectory, we obtain for any $x \in X$ and $u\in\Uc$:
\begin{align*}
\dot{W}_u (x) &\leq \sum_{k=1}^{\infty} \frac{2^{-k}}{1+M(k,k)} \dot{V}_k (x) \\
 &\leq   \sum_{k=1}^{\infty} \frac{2^{-k}}{1+M(k,k)} \Big(V_k(x) + \frac{1}{k}\Big) = W(x) + C_2,
%\label{eq:WdotEstimate}
\end{align*}
for a certain $C_2>0$.
\end{proof}

%\begin{proof}
%The only significant change is a construction of a Lyapunov function in the implication 
%(iv) $\qrq$ (v).
%
%Now we consider $V_k:X\to \R_+$, $k\in \N$, defined for all $x\in X$ as follows:
%\begin{align}
%\label{eq:LF-RFC}
%V_k(x):=\sup_{u \in \Dc}\sup_{t\geq 0}G_k\Big(e^{-t}\rho\Big(\frac{1}{3}\|\phi(t,x,u)\|_X\Big)\Big).
%\end{align}
%In contrast to \eqref{eq:LF-BRS}, we take supremum over all $u\in\Dc$ and use the weighting coefficient $\frac{1}{3}$ as we have only three terms on the right-hand side of (iv).
%As an \q{RFC Lyapunov function} we can take again \eqref{eq:LF_ConverseLyapTheorem_integralType_W-LF}.
%\end{proof}
%
%

\section{Discussion}

\subsection{Relation to finite-dimensional results}

Having proved a characterization of robust forward completeness for a general class of infinite-dimensional control systems, it is of virtue to see how much it can tell us in the special case of ODE systems, and in particular, how Theorem~\ref{thm:RFC-characterization} relates to Lyapunov characterization of forward completeness derived in \cite[Theorem 2]{AnS99}.

Let $\Sigma$ be an ODE system
\begin{eqnarray}
\dot{x} = f(x,u),
\label{eq:ODE-AnS99-comparison}
\end{eqnarray}
where $x(t) \in X:=\R^n$, $u \in \Uc:=L^\infty(\R_+,\R^m)$, and the nonlinearity is as follows:
\begin{ass}
\label{ass:regularity-rhs-ODE}
$f$ is continuous on $\R^n\tm\R^m$ and is Lipschitz continuous in $x$ on bounded sets.
\end{ass}

This assumption ensures that for any initial condition $x \in X$ and any input $u\in\Uc$, the corresponding maximal solution (in the sense of Caratheodory) $\phi(\cdot,x,u)$ of \eqref{eq:ODE-AnS99-comparison} exists and is unique on a certain finite interval. 
Furthermore, $\Sigma:=(X,\Uc,\phi)$ is a well-defined control system with the BIC property, see \cite[Theorem 1.16, Proposition 1.20]{Mir23}. 

Let $R$ be finite, and let $\Dc$ be the set of disturbances as defined by \eqref{eq:disturbance-set}.

Recall that a map $f:\R^n \to \R_+$ is called \emph{proper} if the preimage of any compact subset of $\R_+$ is compact in $\R^n$.

For systems \eqref{eq:ODE-AnS99-comparison}, Theorem~\ref{thm:RFC-characterization} takes the form
\begin{proposition}
\label{prop:RFC-characterization-ODE}
Let Assumption~\ref{ass:regularity-rhs-ODE} hold. 
System \eqref{eq:ODE-AnS99-comparison} is forward complete if and only if there exist a proper Lipschitz continuous function $V:\R^n\to\R_+$ and $a,M>0$ such that the following exponential
growth condition holds:
\begin{align}
\label{eq:Exp-growth-ODE}
\dot{V}_u(x)\leq aV(x) + M,\quad  x \in\R^n,\ u \in\Dc.
\end{align}
\end{proposition}

\begin{proof}
By \cite[Corollary A.11]{Mir23}, $V\in C(\R^n,\R_+)$ is proper if and only if there is $\psi_1\in\Kinf$, such that $V(x) \geq \psi_1(|x|)$ for all $x\in\R^n$.
Furthermore, $V(x) \leq \omega(|x|)$, where $\omega: r \mapsto\sup_{|y|\leq r}V(y)$ is a continuous nondecreasing function. 
Setting $\psi_2(r):=r+\omega(r) - \lim_{s\to +0}\omega(s)$, we obtain 
$V(x) \leq \psi_2(|x|) + \lim_{s\to +0}\omega(s)$.
Thus, $V$ is proper if and only if the \q{sandwich bounds} \eqref{eq:RFC-sandwich-bounds} hold.

In view of Assumption~\ref{ass:regularity-rhs-ODE}, \cite[Proposition 2.5]{LSW96} shows that \eqref{eq:ODE-AnS99-comparison} is forward complete if and only if \eqref{eq:ODE-AnS99-comparison} is RFC.

\q{$\Leftarrow$}. Follows from the above argument and Theorem~\ref{thm:RFC-characterization}.

\q{$\Rightarrow$}. As \eqref{eq:ODE-AnS99-comparison} is robustly forward complete and Assumption~\ref{ass:regularity-rhs-ODE} 
holds, \cite[Lemma 4.6]{MiW19a} ensures that the flow $\phi$ is Lipschitz continuous on compact intervals uniformly in inputs from $\Dc$. 
The rest follows from Theorem~\ref{thm:RFC-characterization}.
\end{proof}

Proposition~\ref{prop:RFC-characterization-ODE} is a version of \cite[Theorem 2]{AnS99}. 
Both results guarantee for a forward complete ODE system the existence of a Lyapunov function with at most exponential growth rate.
However, the Lyapunov function constructed in \cite[Theorem 2]{AnS99} satisfies \eqref{eq:Exp-growth-ODE} with $M=0$, while in our construction $M>0$.
Another difference is that in Proposition~\ref{prop:RFC-characterization-ODE} our Lyapunov function $V$ is Lipschitz continuous, while in 
\cite[Theorem 2]{AnS99} the existence of an infinitely differentiable Lyapunov function with the same properties is shown. 

Basically, the authors in \cite[Theorem 2]{AnS99} construct first a Lipschitz continuous Lyapunov functional and afterward apply the smoothing procedure based on results in \cite{LSW96}, which is developed specifically for ODE systems.
At the same time, in Proposition~\ref{prop:RFC-characterization-ODE}, it is not required that $f$ is Lipschitz continuous with respect to inputs, which is assumed in \cite{LSW96, AnS99}. Furthermore, the argument in \cite{AnS99} uses backward continuation of solutions for ODEs, which is not available for general infinite-dimensional systems.

\subsection{Evolution equations with Lipschitz nonlinearities}

Let us specialize Theorem~\ref{thm:RFC-characterization} to a particular class of infinite-dimensional systems that covers 
many important evolution PDEs with distributed inputs.

Assume that the state space $X$ is a Banach space, the set of input values $U$ is a normed
linear space, and the input functions belong to the space
$\Uc:=PC_b(\R_+,U)$ of globally bounded, piecewise continuous functions $u:\R_+ \to U$, which are right-continuous. The norm of $u \in \Uc$ is given by
$\|u\|_{\Uc}:=\sup_{t \geq 0}\|u(t)\|_U$.

Let $A$ be the generator of a strongly continuous semigroup $T$ of bounded linear operators on $X$.
 Consider the system
\begin{equation}
\label{InfiniteDim}
\dot{x}(t)=Ax(t)+f(x(t),u(t)), \quad t>0,
\end{equation}
where $x(0)\in X$, $u\in \Uc$, and the following assumption holds:
\begin{ass}
\label{ass:Rhs-properties}
The map $f:X \tm U \to X$ satisfies
\begin{enumerate}[label=(\roman*)]  
    \item $f$ is Lipschitz continuous on bounded
subsets of $X$, uniformly with respect to the second argument, i.e., for
all $r>0$, there is $L(r)>0$, such that for all $x,y \in B_r$ and all $v \in B_{r,U}$, it holds that
\begin{eqnarray}
\|f(x,v)-f(y,v)\|_X \leq L(r) \|x-y\|_X.
\label{eq:Lipschitz}
\end{eqnarray}
    \item $f(x,\cdot)$ is continuous for all $x \in X$.
\end{enumerate}
\end{ass}

We study mild solutions of \eqref{InfiniteDim}, i.e., solutions\linebreak $x:[0,\tau] \to X$ of the integral equation
\begin{align}
\label{InfiniteDim_Integral_Form}
x(t)=T(t)x(0) + \int_0^t T(t-s)f(x(s),u(s))ds,
\end{align}
belonging to the space of continuous functions $C([0,\tau],X)$ for some $\tau>0$.

Under such assumptions, for any initial condition $x \in X$ and any input $u\in\Uc$, there is a unique maximal (mild) solution
of \eqref{InfiniteDim}, which we denote by $\phi(\cdot,x,u)$. Moreover, the triple $(X,\Uc,\phi)$ is a control system in the sense of our definition satisfying BIC property. This follows from general results in \cite{Mir23d}.
%, but for the special setting considered here, this can be shown by a variation of the classical results \cite{Paz83,CaH98}.

For \eqref{InfiniteDim}, we can restate Theorem~\ref{thm:RFC-characterization} with assumptions on $f$ rather than on the properties of the flow $\phi$:
\begin{corollary}
\label{cor:RFC-characterization-EE} 
Let $f$ satisfy Assumption~\ref{ass:Rhs-properties}.
Then the assertions (i)--(v) of Theorem~\ref{thm:RFC-characterization} are equivalent. 
%If any of these assertions holds, then $\phi$ is Lipschitz continuous on compact intervals uniformly in inputs from $\Dc$. 
\end{corollary}

\begin{proof}
By \cite{Mir23d}, \eqref{InfiniteDim} is a control system satisfying BIC property. 
Assumption~\ref{ass:Concatenation-in-D} is evidently satisfied.

The assumption that $\phi$ is Lipschitz continuous on compact intervals uniformly in inputs from $\Dc$ was used in 
Theorem~\ref{thm:RFC-characterization} only for the implication (iv) $\Rightarrow$ (v). 
However, as shown in \cite[Section 3.5]{Mir23d}, if \eqref{InfiniteDim} satisfies the RFC property, then $\phi$ is Lipschitz continuous on compact intervals uniformly in inputs from $\Dc$, and hence the invocation of Theorem~\ref{thm:RFC-characterization} shows the claim.
\end{proof}

\subsection{Remarks on boundedness of reachability sets}

The following property is closely related to robust forward completeness and is frequently used in control theory:
\begin{definition}
\label{def:BRS}
\index{bounded reachability sets}
\index{BRS}
We say that \emph{$\Sigma=(X,\Uc,\phi)$ has bounded reachability sets (BRS)} if it is forward complete and for any $r>0$ 
and any $\tau>0$, it holds that 
\[
\sup\big\{
\|\phi(t,x,u)\|_X : \|x\|_X\leq r,\ \|u\|_{\Uc} \leq r,\ t \in [0,\tau]\big\} < \infty.
\]
\end{definition}

It is not hard to see that a control system $\Sigma$ has BRS if and only if it is robustly forward complete with respect to $\Dc$ defined in \eqref{eq:disturbance-set} for all $R<\infty$.
Thus, the BRS property is (in general) stronger than RFC with respect to $\Dc$ with a fixed finite $R$. At the same time, BRS is generally weaker than RFC with $\Dc=\Uc$.

It is reasonable to ask whether one can obtain a Lyapunov characterization of the BRS property as well. 
The following result proposes a natural candidate for a \q{BRS Lyapunov function}:

%\mir{Why do I need Lipschitz continuity next?}

\begin{proposition}
\label{prop:BRS-Direct-Lyapunov-Theorem} 
Consider a control system $\Sigma=(X,\Uc,\phi)$ satisfying the BIC property.

Let there exist a continuous map $V:X\to\R_+$, maps $\psi_1,\psi_2 \in\Kinf$, and $C>0$ such that 
\begin{eqnarray}
\psi_1(\|x\|_X) \leq V(x) \leq \psi_2(\|x\|_X) + C,\quad x\in X,
\label{eq:BRS-sandwich-bounds}
\end{eqnarray}
and there are $a>0$ and $\gamma\in\Kinf$, such that for all $x \in X$ and $u\in\Uc$ the following holds:
\begin{align}
\label{eq:BRS-dissipation}
\|x\|_X \ge \gamma(\|u\|_{\Uc}) \qrq \dot{V}_u(x) \leq a V(x).
\end{align}
Then $\Sigma$ has bounded reachability sets.
\end{proposition}

\begin{proof}
Pick any $x \in X$ and any $u\in\Uc$. As $\Sigma$ is a well-defined control system, there is a maximal time $t_m(x,u)$ such that $\phi(\cdot,x,u)$ is well-defined on $[0,t_m(x,u))$.

Take any finite $\tau \leq t_m(x,u)$, and define
\[
P:=\{t\in(0,\tau): \|\phi(t,x,u)\|_X > \gamma(\|u\|_{\Uc})\}.
\]
By the identity axiom, $\phi(0,x,u)=x$ for all $u$. Together with the definition of $P$, for $t \in[0,\tau)\backslash P$ we have
\begin{align}
\label{eq:BRS-bound-2}
\|\phi&(t,x,u)\|_X \leq \max\{\|x\|_X,\gamma(\|u\|_{\Uc})\}.
\end{align}
If $P\neq\emptyset$, take any $t \in P$, and consider the maximal (w.r.t. the set inclusion) open interval $I=(t_-,t_+) \subset P$, such that $t \in I$. 

By continuity of $\phi(\cdot,x,u)$, such an interval is well-defined, and either 
$\|\phi(t_-,x,u)\|_X = \gamma(\|u\|_{\Uc})$, 
or $t_-=0$, and by the identity axiom $\|\phi(t_-,x,u)\|_X=\|x\|_X$.

By the axiom of shift invariance, $\|u\|_{\Uc} \geq \|u(\cdot+r)\|_{\Uc}$ for any $r \geq 0$, and we have that 
\[
\|\phi(t,x,u)\|_X > \gamma(\|u(t+\cdot)\|_{\Uc}),\quad t\in I.
\]
By \eqref{eq:BRS-dissipation}, for all $t\in [t_-,t_+)$ we have that 
\begin{align*}
%\label{eq:RFC-dissipation-with-time}
 D^+V(\phi(t,x,u)) =\dot{V}_{u(\cdot+t)}(\phi(t,x,u)) \leq a V(\phi(t,x,u)).
\end{align*}
Employing Proposition~\ref{prop:ComparisonPrinciple-exp} for the continuous map 
\[
y(t):=V(\phi(t,x,u)),\quad t\in[t_-,t_+),
\]
we obtain:
\[
V(\phi(t,x,u)) \leq e^{a(t-t_-)}V(\phi(t_-,x,u)),\quad  t\in[t_-,t_+).
\]
Thanks to \eqref{eq:BRS-sandwich-bounds}, we have for all $t\in[t_-,t_+)$ that 
\begin{align*}
%\label{eq:}
\psi_1(\|\phi(t,x,u)\|_X) \leq e^{a(t-t_-)} \Big(  \psi_2(\|\phi(t_-,x,u)\|_X) + C\Big).
\end{align*}
Thus, for all $t\in P$
\begin{align}
\label{eq:BRS-bound-1}
\|\phi&(t,x,u)\|_X \nonumber\\
&\leq \psi_1^{-1}\Big( e^{a \tau} \Big(  \psi_2\big(\max\{\|x\|_X, \gamma(\|u\|_{\Uc})\}\big) + C\Big)\Big).
%&\leq \psi_1^{-1}\Big( e^{a (t-t_-)} \Big(  \psi_2(\max\{\|x\|_X, \gamma(\|u\|_{\Uc})\}) + C\Big)\Big).
\end{align}
Together with \eqref{eq:BRS-bound-2}, this shows that the trajectory $\phi(\cdot,x,u)$ is uniformly bounded on $[0,\tau)$. If $t_m(x,u)$ is finite, taking $\tau:=t_m(x,u)$, we obtain a contradiction to the BIC property. Hence, $t_m(x,u)=+\infty$, and \eqref{eq:BRS-bound-2} and \eqref{eq:BRS-bound-1} show the BRS property.
\end{proof}

A possible approach to obtaining the converse Lyapunov theorem for the BRS property is to transform the control system 
$\Sigma$ with BRS into an RFC auxiliary system $\tilde{\Sigma}$ by using the state feedback $u(x)=d(t)k(x)$, where $d$ is understood as a disturbance belonging to the bounded closed ball of a fixed radius, and $k$ is a carefully chosen feedback law. 
Using RFC characterization (Theorem~\ref{thm:RFC-characterization}) shown in this work, one obtains the RFC Lyapunov function for the modified system $\tilde{\Sigma}$. It could be used to obtain a BRS Lyapunov function for the original system $\Sigma$. The modification method was successfully employed in \cite{SoW95} for the characterization of the ISS property for ODE systems and in \cite{AnS99} for the characterization of the so-called unboundedness observability property for ODE systems with outputs, as well as for the BRS property for ODE systems. 

Employing this method for the systems considered in this work raises several challenges.
One of them is that the abstract systems used in this work are defined in terms of the flow map. 
Thus, there is no trivial way to define explicitly the modified system that will be obtained after adding feedback.
Infinite-dimensionality adds additional complexities as a question appears, whether such a feedback makes an auxiliary closed-loop system well-posed. 
These interesting problems are left for future research.

\subsection*{Acknowledgements}

The author thanks Iasson Karafyllis for his insightful comments and enlightening discussions on an early version of this paper 
as well as anonymous reviewers for their helpful comments on the first version of this work.

\bibliographystyle{abbrv}
%\bibliographystyle{ifacconf}

%\bibliography{C:/Users/Andrii/Dropbox/TEX_Data/Mir_LitList_NoMir,C:/Users/Andrii/Dropbox/TEX_Data/MyPublications}

%\bibliography{C:/GoogleDrive/TEX_Data/Mir_LitList_NoMir,C:/GoogleDrive/TEX_Data/MyPublications}
%\bibliography{MyPublications,Mir_LitList_NoMir}

\begin{thebibliography}{10}

\bibitem{AnS99}
D.~Angeli and E.~D. Sontag.
\newblock Forward completeness, unboundedness observability, and their
  {L}yapunov characterizations.
\newblock {\em Systems \& Control Letters}, 38(4-5):209--217, 1999.

\bibitem{Had72}
J.~R. Haddock.
\newblock Liapunov functions and boundedness and global existence of solutions.
\newblock {\em Applicable Analysis}, 2(4):321--330, 1972.

\bibitem{Hen81}
D.~Henry.
\newblock {\em Geometric Theory of Semilinear Parabolic Equations}.
\newblock Springer, Berlin, 1981.

\bibitem{Iwa83}
T.~Iwamiya.
\newblock Global existence of solutions to nonautonomous differential equations
  in {B}anach spaces.
\newblock {\em Hiroshima Mathematical Journal}, 13(1):65--81, 1983.

\bibitem{JMP20}
B.~Jacob, A.~Mironchenko, J.~R. Partington, and F.~Wirth.
\newblock Noncoercive {L}yapunov functions for input-to-state stability of
  infinite-dimensional systems.
\newblock {\em SIAM Journal on Control and Optimization}, 58(5):2952--2978,
  2020.

\bibitem{Jus67}
A.~Juscenko.
\newblock Necessary and suficient conditions for the global existence of
  solutions of systems of differential equations.
\newblock {\em Dok. Akad. Nauk BSSR}, 11:867--869, 1967.

\bibitem{Kar04}
I.~Karafyllis.
\newblock The non-uniform in time small-gain theorem for a wide class of
  control systems with outputs.
\newblock {\em European Journal of Control}, 10(4):307--323, 2004.

\bibitem{KaJ11b}
I.~Karafyllis and Z.-P. Jiang.
\newblock {\em Stability and Stabilization of Nonlinear Systems}.
\newblock Springer, London, 2011.

\bibitem{KPC22}
I.~Karafyllis, P.~Pepe, A.~Chaillet, and Y.~Wang.
\newblock Is global asymptotic stability necessarily uniform for time-invariant
  time-delay systems?
\newblock {\em SIAM Journal on Control and Optimization}, 60(6):3237--3261,
  2022.

\bibitem{KaS67}
J.~Kato and A.~Strauss.
\newblock On the global existence of solutions and {L}iapunov functions.
\newblock {\em Annali di Matematica Pura ed Applicata}, 77(1):303--316, 1967.

\bibitem{LSW96}
Y.~Lin, E.~D. Sontag, and Y.~Wang.
\newblock A smooth converse {L}yapunov theorem for robust stability.
\newblock {\em SIAM Journal on Control and Optimization}, 34(1):124--160, 1996.

\bibitem{Mir23}
A.~Mironchenko.
\newblock {\em Input-to-State Stability: Theory and Applications}.
\newblock Springer Nature, 2023.

\bibitem{Mir23d}
A.~Mironchenko.
\newblock Well-posedness and properties of the flow for semilinear evolution
  equations.
\newblock {\em Submitted}, 2023.

\bibitem{MiI16}
A.~Mironchenko and H.~Ito.
\newblock Characterizations of integral input-to-state stability for bilinear
  systems in infinite dimensions.
\newblock {\em Mathematical Control and Related Fields}, 6(3):447--466, 2016.

\bibitem{MiW18b}
A.~Mironchenko and F.~Wirth.
\newblock Characterizations of input-to-state stability for
  infinite-dimensional systems.
\newblock {\em IEEE Transactions on Automatic Control}, 63(6):1602--1617, 2018.

\bibitem{MiW19a}
A.~Mironchenko and F.~Wirth.
\newblock Non-coercive {L}yapunov functions for infinite-dimensional systems.
\newblock {\em Journal of Differential Equations}, 105:7038--7072, 2019.

\bibitem{Paz83}
A.~Pazy.
\newblock {\em Semigroups of Linear Operators and Applications to Partial
  Differential Equations}.
\newblock Springer, New York, 1983.

\bibitem{SoW95}
E.~D. Sontag and Y.~Wang.
\newblock On characterizations of the input-to-state stability property.
\newblock {\em Systems \& Control Letters}, 24(5):351--359, 1995.

\bibitem{Tan92}
T.~Taniguchi.
\newblock Global existence of solutions of differential inclusions.
\newblock {\em Journal of Mathematical Analysis and Applications},
  166(1):41--51, 1992.

\bibitem{Wei89b}
G.~Weiss.
\newblock Admissibility of unbounded control operators.
\newblock {\em SIAM Journal on Control and Optimization}, 27(3):527--545, 1989.

\bibitem{Win45}
A.~Wintner.
\newblock The non-local existence problem of ordinary differential equations.
\newblock {\em American Journal of Mathematics}, 67(2):277--284, 1945.

\end{thebibliography}
% Important: no space in the list.

\end{document}